\documentclass{aims-walt}
\usepackage{amsmath}
  \usepackage{paralist}

\usepackage{enumerate,amsthm,amssymb,bbm,mathrsfs} 

 \usepackage[colorlinks=true]{hyperref}
\hypersetup{urlcolor=blue, citecolor=red}

\newtheorem{theorem}{Theorem}[section]

\newtheorem{lemma}[theorem]{Lemma}
\newtheorem{proposition}{Proposition}

\theoremstyle{definition}
\newtheorem{definition}[theorem]{Definition}
\newtheorem{remark}{Remark}

\newcommand{\ep}{\varepsilon}

\newcommand{\calg}{\mathcal{G}}
\newcommand{\calh}{\mathcal{H}}
\newcommand{\cala}{\mathcal{A}}
\newcommand{\calb}{\mathcal{B}}
\newcommand{\cald}{\mathcal{D}}
\newcommand{\W}{\mathcal{W}}

\newcommand{\bbc}{\mathbb{C}}
\newcommand{\bbn}{\mathbb{N}}
\newcommand{\bbz}{\mathbb{Z}}

\newcommand{\lip}{\langle}
\newcommand{\rip}{\rangle}

\newcommand{\lotg}{{L^2([0,T];\calg)}}
\newcommand{\hkg}{{L^2([0,T];\calg)}}
\newcommand{\hkpg}{{H^{1}([0,T];\calg)}}

\DeclareMathOperator{\sgn}{sgn}
\DeclareMathOperator{\re}{Re}

\numberwithin{equation}{section}

\title[Viscoelastic Control and Source] 
      {Source Reconstruction and Stability via Boundary Control of Abstract Viscoelastic Systems}

\author[Walton Green and Shitao Liu]{}

\subjclass{Primary: 93B07, 35R30; Secondary: 74D05, 35L10.}
 \keywords{inverse source problem, viscoelasticity, boundary control, observability, reconstruction formula, moment method.}

 \email{awgreen@clemson.edu}
 \email{liul@clemson.edu}

\thanks{The second author is supported in part by a Clemson Support for Early Exploration and Development (CU SEED) grant.}

\thanks{$^*$ Corresponding author: Walton Green}

\begin{document}
\maketitle

\centerline{\scshape Walton Green and Shitao Liu}
\medskip
{\footnotesize
 \centerline{School of Mathematical and Statistical Sciences}
   \centerline{Clemson University}
   \centerline{ Clemson, SC 29634, USA}
} 

\bigskip


\begin{abstract}
We study the inverse source problem for a class of viscoelastic systems from a single boundary measurement in a general spatial dimension. We give specific reconstruction formula and stability estimate for the source in terms of the boundary measurement. Our approaches rely on the exact boundary controllability of the corresponding viscoelastic systems for which we also provide a new proof based on a modification of the well-known moment method.
\end{abstract}

\section{Introduction and Main Results}

In this paper, we investigate the exact boundary controllability of viscoelastic systems in a general spatial dimension and apply it to solving the inverse problem of reconstructing a source term in the system by a single boundary measurement. We consider a general setting and give reconstruction formulas as well as stability estimates for the source term. The motivation of our paper is \cite{pandolfi13} in which a similar problem is studied, but only for the viscoelastic wave equation in one spatial dimension, due to the nature of the proof for the corresponding controllability. Based on our recent work \cite{green2019boundary}, in this paper we provide a new proof of the exact controllability for an abstract viscoelastic system in an arbitrary spatial dimension that could accommodate more general settings. After establishing the controllability of the system, we give a reconstruction formula and stability estimates for the inverse problem of recovering a source term in the system from a single boundary measurement. The particular method we use for the source reconstruction based on the boundary controllability of the system is originated by Yamamoto in \cite{yamamoto95} for hyperbolic equations.

To formulate our problem, let $\mathcal{G}$, $\mathcal{H}$ be Hilbert spaces and let
$\cala: \cald(\cala)\subset\calh \to \calh$ be a self-adjoint operator satisfying appropriate assumption (A) below.
We consider the following viscoelastic system for $w: [0,T] \to \calh$:
	\begin{equation} \label{eq:wave}\left\{ \begin{array}{lr}
	\displaystyle w''(t)+\cala w(t) = \int_0^t M(t-s) \cala w(s) \, ds
	& t\in [0,T]\\[3mm]
	w(0)=w_0\quad w'(0)=w_1 \\[1mm]
	\end{array} \right.
	\end{equation}
with the memory kernel $M \in H^2(0,T)$, and $w_0, w_1$ being the initial conditions. 
In this abstract setup, the boundary conditions will be contained in $\cald(\cala)$. We also introduce the observation operator $\calb : \cald(\calb) \subset \calh \to \calg$ satisfying the assumption (B) below.
More precisely, throughout this paper we will assume the following conditions on the linear operators $\cala$ and $\calb$:
\begin{itemize}
	\item[(A)] Let $\cala : \cald(\cala) \subset \calh \to \calh$ be self-adjoint, closed with dense range, having compact resolvent, and {\it semibounded}, i.e., 
	\[ \lip \cala u,u \rip \ge -c\|u\|^2, \]
	for some $c>0$ and all $u \in \cald(\cala)$. We will denote by $\calh^1$ the completion of $\cald(\cala)$ with respect to the norm $\|x\|_1^2 := \|x\|^2 + \| |\cala|^{1/2}x\|^2 $.
\medskip
	\item[(B)] Let $\calb : \cald(\calb) \subset \calh \to \calg $ be closed with dense range that satisfies the following {\it observability-regularity} inequality: There exists $T_0 \ge 0$ such that for any $T > T_0$, there exists $C>0$ such that for $w$ satisfying (\ref{eq:wave}) with $M=0$,
	\begin{equation}\label{eq:ob} C^{-1}\|(w_0,w_1)\|_{\calh^1 \times \calh} \le  \|\calb w\|_\hkg \le C \|(w_0,w_1)\|_{\calh^1 \times \calh} \end{equation}
	for all $(w_0,w_1) \in \calh^1\times\calh$ (In the rest of the paper we also use the form $\|\calb w\|_\hkg \asymp \|(w_0,w_1)\|_{\calh^1 \times \calh}$ at times to denote a two-sided inequality like (\ref{eq:ob})).
	\end{itemize}

The condition (A) above is standard, satisfied by most unbounded, self-adjoint, elliptic differential operators on bounded domains. See below for a few concrete examples. We make a few remarks, however, about the condition (B). First, contained in the lower inequality of (\ref{eq:ob}) is the observability inequality of the system (\ref{eq:wave}) \textit{without the memory term}. In this paper, we do not claim to have any new insights concerning this problem. We refer the readers to the extensive literature concerning the controllability and observability of hyperbolic type equations and systems.

We briefly summarize two cases in which the conditions (A) and (B) are both satisfied (therefore our Theorems \ref{thm:1}-\ref{thm:formula} below apply). First, the Dirichlet viscoelastic wave equation defined on an open bounded domain $\Omega$ with smooth boundary where $\cala = -\Delta + q$ with $q$ being a bounded potential. Taking $\calb$ as the Neumann trace on a suitable portion of boundary $\partial\Omega$ (see \cite{b-l-r} for sharp conditions), the condition (B) is well known to be satisfied, see for example \cite{lasiecka86,lions88,zhang00}, with $\calh = L^2(\Omega)$.

Another case to which our result applies is the viscoelastic plate equation where $\cala = \Delta^2$ with Dirichlet boundary condition. It has been considered in \cite{lasiecka-plate} under a smallness assumption on the memory kernel and in \cite{pandolfi-plate} in dimension two. The observability-regularity inequality (\ref{eq:ob}) can be found in \cite[Remark 1.3]{lasiecka89euler} when $\calb$ is the third-order boundary trace and $\calh = H^1_0(\Omega)$. As shown in \cite{lions88,zuazua87}, (\ref{eq:ob}) still holds for the second-order boundary trace with $\calh = L^2(\Omega)$. Moreover, in both cases, $T_0=0$ so the viscoelastic plate equation we consider is still observable in arbitrary time $T>0$.

A point of interest for us is the \textit{Neumann viscoelastic} control and observation problem (e.g., take $\cala$ to be the Neumann Laplacian). To the best of our knowledge, this has not been studied in the literature, and it would be a consequence of our Theorem \ref{thm:1} below, except that it is not known if there are suitable spaces $\calh$, $\calg$ and operator $\calb$ satisfying the condition (B). For the natural choice of $\calb$ as the Dirichlet trace in the case of wave equations, the closest to (\ref{eq:ob}) to our best knowledge is Theorem 2.1.1 in \cite{ltz} for the lower inequality and Theorem 1.1 in \cite{lt89neumann} for the upper inequality.

In this paper we will first study the {\it exact observability} of the system (\ref{eq:wave}) (with $M\neq 0$). That is, 
we are interested in whether the phenomenon $w$ can be observed by the operator $\calb$, i.e. proving the {\it observability inequality}
	\begin{equation}\label{1.5} \|(w_0,w_1)\|_{\calh^1 \times \calh} \le C \|\calb w\|_\hkg \end{equation}
for some $C,T>0$ independent of $(w_0,w_1)$. By duality, it is well known that the exact observability is equivalent to the exact controllability of the dual system to (\ref{eq:wave}) \cite{lions88,russell78}. 
Once we establish the observability/controllability, we will apply that to solve the corresponding inverse source problem from a single measurement. 

Solving an inverse problem through the observability/controllability of the underlying system is a well established technique and has produced various methods in inverse problems. In particular, the celebrated Boundary Control method pioneered by Belishev \cite{belishev87} which deals with the so-called many measurements formulation \cite{isakov06}. For our inverse problem with a single measurement formulation, we refer to \cite{Liu-T13} and references therein.  

Inverse source problems for partial differential equations have also been studied extensively in the literature \cite{bellassoued17,isakov06}. For the viscoelastic inverse problem considered in this paper, \cite{cavaterra06} and \cite{loreti17} studied more general viscoelastic equations and showed similar stability estimates by means of Carleman estimates. However, their method does not produce the reconstruction formula as we have in Theorem \ref{thm:formula}.

On the other hand, the controllability of viscoelastic systems has been well studied and we refer to the book \cite{pandolfi-book} for the extensive literature. However, most of the available results are limited to the case where the spatial dimension is one, with the exceptions being \cite{kim93,pandolfi15}.
This is largely due the fact that in one dimension solutions of (\ref{eq:wave}) may be approximated by sums of complex exponentials $\{e^{int}\}$ which are very well understood.
For instance, the treatment in \cite{avdonin13,loreti12} follows this approach using the well known \textit{moment method} of Russell \cite{russell78}. 

In our recent work \cite{green2019boundary}, we provided a new proof of (\ref{1.5}) in the special case where $\cala$ is given by the Dirichlet Laplacian on an open bounded domain with smooth enough boundary, and $\calb$ is given by the Neumann observation operator from a part of the boundary. The main idea is to view the viscoelastic wave equation as a perturbation of the standard wave equation and show the perturbed harmonic system forms a Riesz sequence. In Section \ref{sec:perturb} below, we prove the analogous result for our abstract viscoelastic system (\ref{eq:wave}) that works for more general settings. 
In other words, we prove that whenever the unperturbed system, namely (\ref{eq:wave}) with $M \equiv 0$, is exactly observable, so is (\ref{eq:wave}) for any $M \in H^2(0, T)$. More precisely, we will show the following:

\begin{theorem}\label{thm:1}
Assume conditions (A) and (B) are satisfied. Let $M \in H^2(0,T)$. Then, for any $T > T_0$, there exists $C>0$ such that (\ref{1.5}) holds for any $(w_0,w_1) \in \calh^1 \times \calh$.
\end{theorem}

After establishing Theorem~1.1, the second problem we study is the reconstruction and stability of an unknown source $f \in \calh$ from the observed data $\calb u\in\calg$ in the following system: let $u : [0,T] \to \calh$ satisfy
\begin{equation} \label{eq:source}\left\{ \begin{array}{lr}
	\displaystyle u''(t)+\cala u(t) = \int_0^t M(t-s) \cala u(s) \, ds + \sigma(t)f 
	& \mbox{in } [0,T]\\[3mm]
	u(0)=0\quad u'(0)=0.  \\[1mm]
	\end{array} \right.
	\end{equation}

The stability estimate of recovering $f$ is given by 
\begin{theorem}\label{thm:stability}
Assume conditions (A) and (B) are satisfied, $M \in H^2(0,T)$, $\sigma \in C^1[0,T]$ with $\sigma(0)\ne 0$, and $T>T_0$. Then there exists $C>0$ such that for any $f \in \calh$, $u$ satisfying (\ref{eq:source}),
	\begin{equation}\label{eq:stability} C^{-1}\|f\|_\calh \le \left\| \calb u \right\|_\hkpg \le C \|f\|_\calh . \end{equation}
\end{theorem}

In addition, we give a reconstruction formula for $f$ in terms of the observation. More specifically, Let $\{\phi_n\}  \subset \calh$ be an orthonormal basis of $\calh$ given as eigenfunctions of $\cala$ (existence is guaranteed by condition (A)). Then we have
\begin{theorem}\label{thm:formula}
Under the assumptions of Theorem \ref{thm:stability}, there exists $\{\theta_n\} \subset \hkg$ such that
	\[ f = \sum_{n=1}^\infty \phi_n \left\lip \calb u' ,\theta_n \right\rip_\hkg \]
for $u$ satisfying (\ref{eq:source}). 
\end{theorem}

The rest of the paper is organized as follows. First, the proof of Theorem \ref{thm:1} and the reformulation via the moment method are presented in Section \ref{sec:perturb}. Section \ref{sec:inverse} relates the inverse source problem to the observed system and Theorems \ref{thm:stability} and \ref{thm:formula} are proved.

\section{Observability of the Viscoelastic System}\label{sec:perturb}

\subsection{Spectral Reformulation}
\label{harm}
By condition (A), $\cala$ has an orthonormal basis of eigenfunctions $\{\phi_n\}_{n=1}^\infty$ for $\calh$ with eigenvalues $\{\mu_n\}_{n=1}^\infty \subset [-c,\infty)$, each of finite multiplicity with $\mu_n \to \infty$. 
Set $\lambda_n = \sgn(n)\sqrt{\mu_{|n|}}$ for each $n \in \bbz \backslash \{0\} =:\bbz_0$. Since $\cala$ is semibounded, $|\Im(\lambda_n)| \le c$ for all $n \in \bbz_0$. We divide $\{\lambda_n\}_{n \in \bbz_0}$ into two classes, indexed by
	\[ J_0 = \{ n \in \bbz_0 : \lambda_n=0\}, \quad J_1 = \bbz_0 \backslash J_0. \]
Then, define
	\begin{equation}\label{psidef} 
		\psi_n = \left\{ \begin{array}{cl} \sgn(n)\calb\phi_{|n|} &\mbox{ for } n \in J_0, \\[3mm]
			  \dfrac{\calb  \phi_{|n|}}{\lambda_n} &\mbox{ for } n \in J_1. \end{array} \right.
	 \end{equation}

The second system of interest is the time component of solutions to (\ref{eq:wave}). For $n \in J_1$, let $z_n$ satisfy the following ordinary differential equation:
		\begin{equation} \left\{ \begin{array}{lr}
		z_n''(t) + \lambda_n^2 z_n(t) = -\lambda_n^2 \displaystyle\int_0^t M(t-s) z_n(s) \, ds & \hspace{4ex}t \in [0,T]\\[3mm]
		z_{n}(0) = 1 \quad \quad z_{n}'(0) = i\lambda_n \\[2mm]
		\end{array} \right. \label{viscoode}
		\end{equation}
For $n \in J_0$, set $z_n(t) = 1+i\sgn(n)t$. Notice when $M \equiv 0$, $z_n(t) = e^{i\lambda_nt}$. We also understand $e^{i\lambda_nt}$ to be $1+i\sgn(n)t$ if $\lambda_n=0$.
The main result of this section is that $\{z_n\psi_n\}$ forms a Riesz sequence in $\lotg$ if $\{e^{i\lambda_n t}\psi_n\}$ does. By Proposition \ref{prop1} below, this is equivalent to Theorem \ref{thm:1}.

\begin{definition}{\em \cite{avdonin95,young01}}\label{def:riesz}
A sequence $\{f_n\}$ in a Hilbert space $\calh$ is said to be a \textit{Riesz sequence} if there exists constants $c,C>0$ such that
	\begin{equation}\label{rfseq}
		c \sum |a_n|^2 \le \left\| \sum a_nf_n \right\|^2 \le C \sum |a_n|^2
	\end{equation}
for all finite sequences $\{a_n\} \subset \bbc$.
In the case when the lower or upper inequality holds, $\{f_n\}$ is said to be a \textit{Riesz-Fischer} or \textit{Bessel} sequence, respectively.
\end{definition}
We remark that throughout the paper, when we say a finite sequence, we mean a sequence with only finitely many non-zero entries.

We now state the relationship between Riesz sequences and the observability-regularity inequality (\ref{eq:ob}).

\begin{proposition}\label{prop1}
The inequality (\ref{eq:ob}) holds for all $(w_0,w_1) \in \calh^1 \times \calh$ and $w$ satisfying (\ref{eq:wave}) if and only if $\{z_n\psi_n\}_{n \in \bbz_0}$, defined by (\ref{psidef}) and (\ref{viscoode}), is a Riesz sequence in $\hkg$, i.e. there exists $c,C >0$ such that
	\begin{equation}\label{riesz} c \sum |a_n|^2 \le \int_0^T \left\| \sum a_n z_n(t)\psi_n\right\|^2_\calg \, dt \le C \sum |a_n|^2\end{equation}
for all finite sequences $\{a_n\} \subset \bbc$.
\end{proposition}
\begin{proof}Let $(w_0,w_1) \in \calh^1 \times \calh$.
We will represent the solution $w$ to (\ref{eq:wave}) by separation of variables. In the space variable, we expand onto $\{\phi_n\}$. There exist $\{\xi_n\}, \{\eta_n\} \in \ell^2$ such that
	\[ w_0 = \sum_{n=1}^\infty \xi_n \phi_n, \quad \quad w_1 = \sum_{n=1}^\infty \eta_n \phi_n. \]
Since $w_0 \in \calh^1$, by the orthonormality of $\{\phi_n\}$,
	\[ \||\cala|^{1/2} w_0\|^2 = \lip w_0,|\cala| w_0 \rip = \left\lip \sum_{n=1}^\infty \xi_n \phi_n , \sum_{n=1}^\infty |\lambda_n^2|\xi_n \phi_n \right\rip = \sum_{n=1}^\infty |\lambda_n \xi_n|^2 \]
therefore $\{ \lambda_n \xi_n \}_{n=1}^\infty\in \ell^2$. Set $\tilde \xi_n = \lambda_n \xi_n$ for $n \in J_1$. Then,
	\[ w_0 = \sum_{n \in J_0 \cap \bbn} \xi_n \phi_n + \sum_{n \in J_1 \cap \bbn} \dfrac{\tilde \xi_n}{\lambda_n} \phi_n. \]
Additionally, we consider the ODE (\ref{viscoode}) to account for the time variable. One can then verify that
	\begin{align}\notag w(t) = \dfrac 12 \sum_{n \in J_1 \cap \bbn} &\left[ \left(\dfrac{\tilde \xi_n}{\lambda_n} - i \dfrac{\eta_n}{\lambda_n}\right) z_n(t) - \left(\dfrac{\tilde \xi_n}{\lambda_{-n}} +i\dfrac{\eta_n}{\lambda_{-n}}\right)z_{-n}(t) \right]\phi_{n} \\
		\label{dualrep}&+ \dfrac 12 \sum_{n \in J_0 \cap \bbn} \left[(\xi_n - i\eta_n)z_n(t)+(\xi_n+i\eta_n)z_{-n}(t)\right] \phi_{n}.\end{align}
Then, setting 
	\[ a_n = \left\{ \begin{array}{cl} \sgn(n)\tilde \xi_{|n|} - i\eta_{|n|} & \mbox{for } n \in J_1, \\[2mm]
			\sgn(n)\xi_{|n|}-i\eta_{|n|} &\mbox{for } n \in J_0, \end{array}\right. \]
we have
	\begin{align*} \sum_{n \in \bbz_0} |a_n|^2 &= \sum_{n \in J_0} |\xi_{|n|}|^2 + \sum_{n \in J_1} |\tilde \xi_{|n|}|^2 + \sum_{ n \in \bbn} |\eta_{|n|}|^2 \\
					&\le 2\left( \sum_{n=1}^\infty |\xi_n|^2 + \sum_{n=1}^\infty |\lambda_n \xi_n|^2 + \sum_{n=1}^\infty |\eta_n|^2\right) \\
					&=2 (\|w_0\|^2 + \||\cala|^{1/2}w_0\|^2 + \|w_1\|^2 )\\[2mm]
					&=2( \|w_0\|_1^2 + \|w_1\|^2).
	\end{align*}
To establish the lower inequality ($\displaystyle\sum_{n \in \bbz_0} |a_n|^2 \ge c \|(w_0,w_1)\|^2_{\calh^1\times \calh}$), notice that
	\[ \sum_{n \in J_1 \cap \bbn} | \tilde \xi_n|^2 = \dfrac 12 \sum_{\bbz_0 \cap \bbn} |\lambda_n \xi_n|^2 + \dfrac 12 \sum_{J_1 \cap \bbn} |\lambda_n \xi_n|^2 \ge \dfrac 12 \||\cala|^{1/2}w_0\|^2 + \dfrac{\displaystyle\min_{n\in J_1}\{|\lambda_n|^2\}}{2} \sum_{J_1 \cap \bbn} |\xi_n|^2. \]
The minimum exists and is non-zero since the only accumulation point of the spectrum of $\cala$ is $\infty$. Therefore,
	\[ \sum |a_n|^2 \asymp \|w_0\|^2_1 + \|w_1\|^2. \]
Using the representation (\ref{dualrep}), we see that $\calb w = \tfrac 12 \sum a_n z_n \psi_n$ (recall $\psi_n$ from (\ref{psidef})) which completes the proof.
\end{proof}

The special case $M=0$ ($z_n(t)=e^{i\lambda_nt}$) shows that our assumption (B) is equivalent to the fact that $\{e^{i\lambda_nt}\psi_n\}$ is a Riesz sequence in $\lotg$.

\subsection{Proof of Theorem \ref{thm:1}}\label{main}
In light of Proposition \ref{prop1}, the following theorem is equivalent to Theorem \ref{thm:1}.

\begin{theorem}\label{viscoriesz}
Let $T_0\ge 0$ be such that $\{e^{i\lambda_nt}\psi_n\}$ is Riesz sequence in $\hkg$ for any $T>T_0$. Then, $\{z_n\psi_n\}$ forms a Riesz sequence in $\hkg$ for any $T > T_0$.
\end{theorem}

Our approach is similar to \cite{avdonin13,loreti12,pandolfi15,green2019boundary} in the sense that we will argue that $\{z_n\psi_n\}$ is in a certain sense ``close'' to $\{e^{i\lambda_n t} \psi_n\}$. We will employ a simple version of the classical Paley-Wiener theorem \cite[p. 38]{young01} on equivalent bases to compare the two sequences.

\begin{lemma}\label{perr}
Let $\{e_n\}$ be a Riesz sequence in a Hilbert space $\calh$ and $\{f_n\} \subseteq \calh$. If there exists $q \in (0,1)$ such that
	\begin{equation}\label{percond} \left\| \sum a_{n}(e_{n}-f_{n}) \right\| \le q \left\| \sum_{n} a_{n}e_{n} \right\| \end{equation}
for all finite sequences $\{a_{n}\}$, then $\{f_n\}$ is also a Riesz sequence.
\end{lemma}

\begin{proof}
By the triangle inequality,
	\[\left\|\sum a_nf_n\right\| \ge \left\|\sum a_ne_n\right\|-\left\|\sum a_n(e_n-f_n)\right\| \ge (1-q)\left\|\sum a_ne_n\right\|. \]
The upper inequality follows in the same way.
\end{proof}

The first step in proving Theorem \ref{viscoriesz} is the following proposition.

\begin{proposition}\label{prop:2}
Let $\{z_n\}$ and $\{\psi_n\}$ be defined by (\ref{psidef}) and (\ref{viscoode}). If $\{e^{i\lambda_n t} \psi_n\}$ is a Riesz sequence, then there exists $N>0$ such that $\{z_n\psi_n\}_{|n| \ge N}$ is a Riesz sequence.
\end{proposition}
We will use the following lemma which shows there is some orthogonality in the sequence $\{\psi_n\}$ as a consequence of (B). 
\begin{lemma}\label{psib}
Suppose $\{e^{i\lambda_nt}\psi_n\}$ is a Bessel sequence in $\lotg$. There exists $C>0$ such that for any $\ep \in (0,T]$ and any finite sequence $\{a_n\} \subset \bbc$,
	\begin{equation}\label{psib2} \left\| \sum a_n \psi_n \right\|^2_\calg \le C\left(\ep^{-1}\sum |a_n|^2 + \ep\sum |\lambda_na_n|^2 \right) \end{equation}
\end{lemma}
\begin{proof}
Let $\{a_n\} \subset \bbc$ be a finite collection of scalars, $\ep \in (0,T]$.
\begin{align*}
	\ep\left\| \sum a_n \psi_n \right\|^2_\calg & = \int_0^\ep \left\|\sum a_ne^{i\lambda_nt}\psi_n-a_n(e^{i\lambda_nt}-1)\psi_n \right\|^2_\calg \, dt \\			
					&\le 2 \int_0^T \left\|\sum a_ne^{i\lambda_nt}\psi_n\right\|^2 \, dt + 2 \int_0^\ep \left\| \sum a_n(e^{i\lambda_nt}-1)\psi_n \right\|^2 \, dt 
	\end{align*}
The first term is bounded by $2C \sum |a_n|^2$ by the Bessel inequality. To estimate the second term, we split the sum over $J_0$ and $J_1$. Since $\int_0^t e^{i\lambda_ns}(i\lambda_n) \, ds = e^{i\lambda_nt}-1$,
	\begin{align*} \int_0^\ep \left\| \sum_{n \in J_1} a_n(e^{i\lambda_nt}-1)\psi_n \right\|^2 \, dt &= \int_0^\ep \left\| \int_0^t \sum_{J_1} e^{i\lambda_n s} (i\lambda_n) \, ds \, a_n \psi_n \right\|^2 \, dt \\
					&\le \int_0^\ep \left(\int_0^t \, ds \right)\left( \int_0^t \left\| \sum_{J_1} e^{i\lambda_n s} (i\lambda_n a_n) \psi_n \right\|^2 \, ds \right)\, dt \\
					&\le \dfrac{\ep^2}{2} \int_0^T \left\| \sum_{J_1} e^{i\lambda_n s} (i\lambda_n a_n) \psi_n \right\|^2 \, ds \\
					&\le \dfrac{C\ep^2}{2} \sum_{J_1} |a_n\lambda_n|^2.  \end{align*}
For $n\in J_0$, set $b_n=a_n+a_{-n}$. Then, since $b_n$ and $\sgn(n)it\psi_n$ are even functions of $n$ and $\psi_n$ is odd,
	\[ \dfrac 12 \sum_{n \in J_0}b_n(1+\sgn(n)it)\psi_n = \sum_{J_0 \cap \bbn} b_n \sgn(n)it \psi_n \]
	\[= \sum_{J_0} a_n(\sgn(n)it)\psi_n = \sum_{J_0} a_n(e^{i\lambda_nt}-1) \psi_n.\]
Therefore,
	\[ \int_0^\ep \left\| \sum_{n \in J_0} a_n(e^{i\lambda_nt}-1) \psi_n \right\|^2 \, dt \le \dfrac{C}{4} \sum_{J_0} |b_n|^2 \le C \sum |a_n|^2 .\]
\end{proof}

\begin{proof}[Proof of Proposition \ref{prop:2}]
As shown in \cite{loreti12,green2019boundary}, there exists $C_1>0$ such that
	\begin{equation}\label{zest} \int_0^T |z_n(t) - e^{(\gamma + i \lambda_n) t}|^2 \, dt \le \dfrac{C_1}{\lambda_n^2} \end{equation}
for $\gamma =M(0)/2$, $\lambda_n \ne 0$. When $\lambda_n=0$, $z_n(t) = 1+i\sgn(n)t=e^{i\lambda_nt}$, so the difference is zero.
Applying Lemma \ref{psib} and then the above estimate (\ref{zest}), we have for any $N \in \bbn$
	\begin{align}
		\notag\int_0^T &\left\| \sum_{|n| \ge N} a_n \psi_n \left(z_n(t)-e^{(\gamma+i\lambda_n )t}\right) \right\|^2 \, dt \\
			\notag&\le C \left( \int_0^T \ep^{-1}\sum_{|n| \ge N} |a_n(z_n-e^{(\gamma+i\lambda_n )t})|^2 + \ep\sum_{|n| \ge N} |\lambda_na_n(z_n-e^{(\gamma+i\lambda_n )t})|^2\right) \, dt \\
			\label{eq:pwest}&\le C C_1 \left((\ep |\lambda_N|^2)^{-1} \sum |a_n|^2 + \ep \sum |a_n|^2 \right).
	\end{align}
Now, since $\{e^{i\lambda_n t} \psi_n\}$ forms a Riesz sequence,
	\begin{align*} \int_0^T \left\|\sum a_ne^{(\gamma+i\lambda_n)t} \psi_n \right\|^2 \,dt &\ge \min\{1,e^{-\re \gamma}\} \int_0^T \left\|\sum a_n e^{i\lambda_nt}\psi_n \right\|^2 \, dt \\
		&\ge c\min\{1,e^{-\re \gamma}\} \sum |a_n|^2. \end{align*}
Setting $c_\gamma = c\min\{1,e^{-\re \gamma}\}$, (\ref{eq:pwest}) implies
	\begin{multline*}\int_0^T \left\| \sum_{|n| \ge N} a_n \psi_n \left(z_n(t)-e^{(\gamma+i\lambda_n )t}\right) \right\|^2 \, dt  \\ 
	\le CC_1 c_\gamma^{-1} ((\ep|\lambda_N|^2)^{-1}+\ep) \left\| \sum a_n e^{(\gamma+i\lambda_n )t} \psi_n \right\|^2. \end{multline*}
Now, take $\ep$ small (so that $CC_1c_\gamma^{-1}\ep = \tfrac 14$). Then, since $|\lambda_n| \to \infty$ pick $N$ large enough that $\tfrac 14 |\lambda_N| \ge CC_1c_\gamma^{-1}\ep^{-1}$. Therefore, by Lemma \ref{perr} with $q=\tfrac 12$, we conclude that $\{z_n\psi_n\}_{|n| \ge N}$ forms a Riesz sequence in $\lotg$.
\end{proof}

The second step in proving Theorem \ref{viscoriesz} is to establish the $\ell^2$-independence of $\{z_n\psi_n\}$.
\begin{proposition}\label{ltwoind}
If $\{e^{i\lambda_nt}\psi_n\}$ is a Riesz sequence in $\lotg$, then the sequence $\{z_n\psi_n\}$ defined by (\ref{psidef}) and (\ref{viscoode}) is $\ell^2$-independent in $\lotg$, i.e. for any $\{a_n\} \in \ell^2$ s.t. $\sum a_nz_n\psi_n=0$, $a_n=0$ for all $n$.
\end{proposition}

Together, Propositions \ref{prop:2} and \ref{ltwoind}, establish that $\{z_n\psi_n\}$ is a Riesz sequence (Theorem \ref{viscoriesz}) by virtue of the following lemma from the Appendix of \cite{green2019boundary}.

\begin{lemma}\label{ltwo}
Let $\{f_n\}_{n=1}^\infty$ be a sequence in a Hilbert space $\calh$. If $\{f_n\}_{n \ge N}$ is a Riesz sequence for some $N \in \bbn$ and $\{f_n\}_{n=1}^\infty$ is $\ell^2$-independent, then $\{f_n\}_{n=1}^\infty$ is a Riesz sequence.
\end{lemma}

It only remains to prove Proposition \ref{ltwoind}.

\begin{proof}[Proof of Proposition \ref{ltwoind}]
Define $e_n(t) = z_n(t)-e^{(\gamma+i\lambda_n)t}$. Since $\sum a_nz_n\psi_n=0$, $\sum a_ne_n(t)\psi_n = -\sum a_ne^{(\gamma+i\lambda_n)t}\psi_n$. As shown in the proof of Proposition 3.6 in \cite{green2019boundary}, for $n \in J_2 :=\{ n \in \bbz_0$ : $\lambda_{|n|}>0\}$,
	\[ e_n'(t) = f_n(t)+O(\lambda_n^{-1}) \]
where $\{f_n\psi_n\}$ forms a Bessel sequence (recall Definition \ref{def:riesz}). Notice that by Lemma \ref{psib}, when the $O(\lambda_n^{-1})$ term is multiplied by $\psi_n$, it will also form a Bessel sequence. Therefore, $\{e_n'\psi_n\}_{n \in J_2}$ forms a Bessel sequence and since $\{a_n\} \in \ell^2$, $\sum_{n \in J_2} a_ne_n'\psi_n$ converges. This can be extended to all $\bbz_0$ since $\bbz_0 \backslash J_2$ is finite so 
	\[ \dfrac{d}{dt} \sum_{n \in \bbz_0} a_ne^{(\gamma+i\lambda_n)t}\psi_n = \sum_{\bbz_0} a_ne_n'\psi_n \in \lotg . \]
By Lemma 3.3 in \cite{pandolfi-book}, this implies that $\{a_n\lambda_n\} \in \ell^2$. Repeating this process for $e_n''$, we get that $\{e_n''\lambda_n^{-1}\psi_n\}_{n \in J_2}$ is a Bessel sequence so $\{a_n\lambda_n^2\} \in\ell^2$. Therefore,
	\[ \sum a_nz_n''\psi_n = \sum a_n(\gamma+i\lambda_n)^2e^{(\gamma+i\lambda_n)\cdot}\psi_n + \sum a_n\lambda_n\dfrac{e_n''}{\lambda_n}\psi_n \in \lotg. \]
This allows us to exchange the derivative and the sum yielding
	\[ 0 = \dfrac{d^2}{dt^2} \sum a_nz_n(t)\psi_n = -\sum a_n\lambda_n^2z_n(t)\psi_n - \int_0^t M(t-s) \sum a_n\lambda_n^2z_n(s) \psi_n \, ds. \]
By standard theory of Volterra integral equations, we have $\sum a_n\lambda_n^2z_n \psi_n = 0$.

Now, set $\Lambda = \{ n \in \bbz : |\lambda_{|n|}| < |\lambda_{|n|+1}| \}$, i.e. the indices corresponding to distinct eigenvalues of $\cala$. For each $n \in \Lambda$, define
	\[\Psi_n = \sum_{\lambda_m=\lambda_n} a_m\psi_m. \]
Then, $0 = \sum_{\bbz_0} a_n\lambda_n^2z_n\psi_n= \sum_\Lambda \lambda_n^2 \Psi_n z_n$. Next,
set $\Psi_n^{(1)} = (\lambda_{1}^2-\lambda_n^2)\Psi_n$. Then, notice that $\Psi_n^{(1)}$ has the following properties:
\begin{itemize}
	\item[(a)] $\displaystyle \sum \Psi_n^{(1)} z_n = \lambda_1^2 \sum \Psi_n z_n(t) - \sum \lambda_n^2\Psi_n z_n(t) = 0 $.
	\item[(b)] $\Psi_1^{(1)} = \Psi_{-1}^{(1)}=0$ but for $|n| > 1,$ $\Psi_{n}^{(1)} = 0 \iff \Psi_n=0$.
\end{itemize}
This can be repeated for $m \in \Lambda$, $2 \le m < N$ by setting $\Psi_n^{(m)} = (\lambda_m^2-\lambda_n^2)\Psi_{n}^{(m-1)}$ (Here $m-1$ means the index in $\Lambda$ immediately preceding $m$). Thus, we have constructed
	\[ \sum_{|n| \ge N}b_n z_n \psi_n= \sum_{\{|n|\ge N\} \cap \Lambda }\Psi_{n}^{(N-1)}z_n = 0 \hspace{5ex} \mbox{with } b_n = a_n\prod_{\substack{1\le k < N, \\ k \in \Lambda}}(\lambda_k^2-\lambda_n^2).\]
But the subsequence $\{z_n\psi_n\}_{|n|\ge N}$ is a Riesz sequence by (ii) so $b_n=0$ which implies $a_n=0$ for $|n|\ge N$. Now we only need to deal with the finite sum
	\begin{equation}\label{linind} \sum_{\{|n| \le N\} \cap \Lambda} \Psi_n z_n= 0 .\end{equation}
We will prove that $\{z_n\}_{|n| \le N \cap \Lambda}$ is linearly independent. If it is not, then there is a smallest linearly dependent subset, indexed by $\{n_k\}_{k=1}^M$, $M \ge 2$, and suitable $\{c_{k}\}$ (non-zero) such that
	\begin{equation}\label{lindep} \sum_{k=1}^M c_{k}z_{n_k}(t)=0.\end{equation}
First, notice that it cannot be the case that both $M=2$ and (\ref{lindep}) is of the form $c_1z_{n_M}(t)+c_2z_{-n_M}(t)=0$. Indeed, $z_n(0)=z_{-n}(0)$ but $z_n'(0)=-z_{-n}'(0)$ (see (\ref{viscoode})) so $c_1=c_2=0$.
Now, we again differentiate twice in time and apply the uniqueness property of the Volterra equation to obtain

$\sum_{k=1}^M \lambda_{n_k}^2c_{n_k}z_{n_k}(t)=0$.
Therefore we have found a smaller linearly dependent collection, namely
	\[\sum_{k=1}^{M} (\lambda_{n_M}^2-\lambda_{n_k}^2)c_{n_k}z_{n_k}(t)=0\]
where one or at most two of the new coefficients are zero (two only if $\lambda_{n_M}$ and $\lambda_{-n_M}$ are in the collection, but then $M>2$). Thus $\{z_n\}$ is linearly independent for distinct $\lambda_n$ and from (\ref{linind}), we conclude that each $\Psi_n=0$. Finally, since $\{e^{i\lambda_nt}\psi_n\}$ forms a Riesz sequence,
	\[ T \max\{1,e^{\Im(\lambda_n)}\} \|\Psi_n\|^2 \ge 
	\int_0^T \left\| \sum_{\lambda_m=\lambda_n} a_me^{i\lambda_mt}\psi_m \right\|^2 \, dt \ge c \sum_{\lambda_n=\lambda_m} |a_m|^2 \]
so each $a_n=0$.
\end{proof}

\section{Inverse Problem}\label{sec:inverse}

\subsection{Stability Estimate}
First we give the relationship between the systems (\ref{eq:wave}) and (\ref{eq:source}).
\begin{lemma}\label{lemma:2} 
Let $w$ satisfy (\ref{eq:wave}) with $w_0=0$, $w_1=f \in \calh$. Then
	\begin{equation}\label{eq:10} u(t) = \int_0^t \sigma (t-s)w(s) \, ds \end{equation}
satisfies (\ref{eq:source}).
\end{lemma}

\begin{proof}
First notice that for any $v \in C^1(0,T)$, integrating by parts, we have
	\begin{align} \label{eq:11a}\dfrac{d}{dt} \int_0^t &\sigma(t-s) v(s) \, ds = \int_0^t \sigma'(t-s) v(s) \, ds + \sigma(0) v(t) \\
	&=-\sigma(t-s)v(s)\Bigr|_{s=0}^{s=t} +\int_0^t \sigma(t-s) v'(s) \, ds + \sigma(0) v(t) \notag \\
	&= \sigma(t)v(0) + \int_0^t \sigma(t-s) v'(s) \, ds \label{eq:11b}\end{align}
Applying this to (\ref{eq:10}), $u$ satisfies the homogeneous initial conditions for (\ref{eq:source}) since $w(0)=w_0=0$.
Differentiating (\ref{eq:10}) with respect to $t$ and applying (\ref{eq:11b}) twice,
	\begin{align*}u''(t) &= \sigma'(t)w(0) + \sigma(t)w'(0) + \int_0^t\sigma(t-s)w''(s) \, ds \\
			&= \sigma(t)f + \int_0^t\sigma(t-s)w''(s) \, ds \end{align*}
where we have used the fact that $w(0) = w_0 = 0$ and $w'(0) = w_1=f$. Next, we claim that
	\begin{align*}\cala u(t) &+ \int_0^t M(t-s) \cala u(s) \, ds = \int_0^t \sigma(t-s) \cala w(s) \, ds \\
				&\hspace{6ex} + \int_0^t\int_0^s M(t-s)\sigma(s-r)\cala w(r) \, dr \, ds \\
			&= \int_0^t\sigma(t-s) \left(\cala w(s) + \int_0^s M(s-r) \cala w(r) \, dr\right) \, ds. 
	\end{align*}
If this holds, then the lemma is proved. We only need to confirm the last step, establishing that the convolutions commute. Indeed,
	\[ \int_0^t\int_0^s M(t-s)\sigma(s-r) v(r) \, dr \, ds = \int_0^t \int_r^tM(t-s)\sigma(s-r) \, ds \, v(r) \, dr \]
	\[= \int_0^t \int_r^{t}M(\tau -r)\sigma(t-\tau) \, d\tau \, v(r) \, dr = \int_0^t \int_0^{\tau}M(\tau -r)\sigma(t-\tau) v(r) \, dr\, d\tau\]
	\[ = \int_0^t \sigma(t-\tau)\int_0^{\tau}M(\tau -r) v(r) \, dr\, d\tau \]
for any $v \in C(0,T)$.
\end{proof}

The stability estimate (Theorem \ref{thm:stability}) is a simple consequence of this lemma.

\begin{proof}[Proof of Theorem \ref{thm:stability}]
As a consequence of Theorem \ref{thm:1}, with $w_0=0$, and $w_1=f$,
	\begin{equation}\label{eq:4} \|f\|_{\calh} \asymp \|\calb w\|_\hkg. \end{equation}
Then, in light of Lemma \ref{lemma:2},
	\begin{equation}\label{eq:12a} u'(x,t) = \sigma(0)w(t) + \int_0^t \sigma'(t-s)w(s) \, ds .\end{equation}
We first prove the lower inequality in (\ref{eq:stability}). By standard theory of Volterra equations \cite{vladimirov76}, there exists $K \in C[0,T]$ (which we will henceforth call the resolvent kernel of $\sigma'/\sigma(0)$) such that
	\begin{equation}\label{eq:12b} \sigma(0)w(t) = u'(t) + \int_0^t K(t-s) u'(s) \, ds. \end{equation}
Note that for any $\rho \in C[0,1],v \in L^2(0,T)$,
	\begin{align*}
		\int_0^T\left|\int_0^t \rho(t-s) v(s) \, ds \right|^2 \, dt & \le \int_0^T\int_0^t |\rho(t-r)|^2 \, dr \int_0^t\left|v(s)\right|^2 \, ds  \, dt \\
			& \le \frac{T^2 \|\rho\|_\infty^2}{2} \int_0^T\left|v(s)\right|^2 \, ds.
	\end{align*}
Applying this to (\ref{eq:12a}) and (\ref{eq:12b}), we obtain
	\begin{equation}\label{eq:12c} \|\calb w\|_\hkg \asymp \|\calb u\|_\hkpg \end{equation}
Applying (\ref{eq:4}) proves the theorem. 
\end{proof}

\subsection{Source Reconstruction}
We will give a formula for the Fourier coefficients of $f$, so we recall the following systems to decompose solutions to (\ref{eq:wave}) and (\ref{eq:source}): $\{\phi_n\}$, $\{\lambda_n\}$, and $\{z_n\}$ from Section \ref{harm}. Also, by Proposition \ref{prop1}, Theorem \ref{thm:1} is equivalent to the fact that $\{z_n\psi_n\}$ forms a Riesz sequence (recall Definition \ref{def:riesz}).

The main property we will use of Riesz sequences is the existence of a biorthogonal Riesz sequence. Two sequences $\{f_n\},\{g_k\}$ are \textit{biorthogonal} to each other if $\lip f_n,g_k \rip = \delta_{n,k}$, where $\delta_{n,k}$ is the Kronecker delta.

\begin{proof}[Proof of Theorem \ref{thm:formula}]
First, since $\{z_n\psi_n\}_{n \in \bbz_0}$ is Riesz sequence in $\lotg$, setting $w_n = \tfrac{z_n-z_{-n}}{2i}$ for $n \in \bbn$, $\{w_n\psi_n\}_{n \in \bbn}$ is still a Riesz sequence. Indeed, for a finite sequence $\{a_n\}_{n \in \bbn} \subset \bbc$,
	\begin{align}
		\notag 2\sum_{n=1}^\infty |a_n|^2 & = \sum_{n \in \bbz_0} |a_{|n|}|^2 \\
			\notag&\asymp \left\| \sum_{n \in \bbz_0} a_{|n|}z_n\psi_n\right\|^2 \\
			\notag&= \left\| \sum_{n=1}^\infty a_{n}z_n\psi_n + \sum_{n=1}^\infty a_{n}z_{-n} \psi_{-n} \right\|^2 \\
			\notag&= \left\| \sum_{n=1}^\infty a_{n}z_n\psi_n - \sum_{n=1}^\infty a_{n}z_{-n} \psi_{n} \right\|^2 \\
			\label{wrs}&= \left\| \sum_{n=1}^\infty 2ia_{n}w_n\psi_n \right\|^2 
	\end{align}
By the formula for $z_n$ (\ref{viscoode}), for $n \in J_1$, $w_n$ satisfies
	\begin{equation} \left\{ \begin{array}{lr}
		w_n''(t) + \lambda_n^2 w_n(t) = -\lambda_n^2 \displaystyle\int_0^t M(t-s) w_n(s) \, ds & \hspace{4ex}t \in [0,T]\\[3mm]
		w_{n}(0) = 0 \quad \quad w_{n}'(0) = \lambda_n \\[2mm]
		\end{array} \right. \label{eq:w-ode}
		\end{equation}
and $w_n(t)=t$ for $n \in J_0$. Since $\{w_n\psi_n\}$ is a Riesz sequence, there exists a biorthogonal Riesz sequence \cite{young01}, which we can compute in the following way. Define $\W:\calh \to \hkg$ by
	\begin{equation}\label{def:t} \W\phi_n := w_n\psi_n\end{equation}
for each $n \in \bbn$ and extend by linearity. $\W$ is bounded and has a bounded inverse using (\ref{wrs}) and the fact that $\{\phi_n\}$ is an orthonormal basis. For each $k \in \bbn$, set $p_k = (\W^{-1})^*\phi_k$. $\{p_k\}$ is biorthogonal to $\{w_n\psi_n\}$ since
	\[ \lip w_n\psi_n,p_k \rip = \lip w_n\psi_n, (\W^{-1})^*\phi_k \rip = \lip \W^{-1}w_n\psi_n,\phi_k \rip = \lip \phi_n,\phi_k \rip = \delta_{n,k}. \]
Next we compute the adjoint of the Volterra operator on $\hkg$, $V_\rho v(t) = \int_0^t\rho(t-s)v(s) \, ds$ for any $\rho \in L^2(0,T)$.
	\[ \int_0^T \int_0^t \rho(t-s)v(s)\, ds z(t) \, dt = \int_0^T \int_s^{T} \rho(t-s) v(s) z(t) \, dt \, ds \]
	\[= \int_0^T v(t) \int_t^T \rho(s-t) z(s) \, ds \, dt \]
So $V_{\rho}^* z(t) = \int_t^T \rho(s-t) \, z(s) \, ds$. We want to find $\theta_k$ such that
	\begin{equation}\label{eq:ptheta3} p_k = (\sigma(0)+V_{\sigma'}^*)\theta_k. \end{equation}
Recalling $K$ from (\ref{eq:12a}) and (\ref{eq:12b}), we see that $(I+V_{K})(\sigma(0)+V_{\sigma'})=\sigma(0)I$ so if we set $\theta_k = \sigma(0)^{-1}(I+V_K^*)p_k$, then (\ref{eq:ptheta3}) is satisfied. Indeed, 
	\[ (\sigma_0+V_{\sigma'}^*)\theta_k = \sigma(0)^{-1}[(I+V_K)(\sigma(0)+V_{\sigma'})]^*p_k = p_k,\]
thus establishing (\ref{eq:ptheta3}).
This gives the reconstruction formula. Indeed, by (\ref{dualrep}) and Lemma \ref{lemma:2}
	\begin{equation}\label{eq:u} u(t) = \int_0^t \sigma(t-s) \sum_{n=1}^\infty a_n w_n(s) \phi_n\, ds \end{equation}
where
	\[ a_n = \left\{ \begin{array}{cl} \lip f,\phi_n \rip & \mbox{ for } n \in J_0\cap \bbn,\\[3mm]
					\dfrac{\lip f,\phi_n \rip}{\lambda_n} &\mbox{ for } n \in J_1 \cap \bbn,
		\end{array} \right. \]
which implies
	\[\calb u' = \calb \sum_{n=1}^\infty a_n  (\sigma(0)+V_{\sigma'})w_n \phi_n = \sum_{n=1}^\infty \lip f, \phi_n \rip (\sigma(0)+V_{\sigma'}) w_n \psi_n. \]
Finally, by (\ref{eq:ptheta3}), for each $k \in \bbn$
	\begin{align*}
		\left \lip \calb u' , \theta_k \right\rip_\hkg &= \sum_{n=1}^\infty \lip f,\phi_n \rip_\lotg \lip (\sigma(0)+V_{\sigma'})w_n \psi_n, \theta_k \rip_\hkg \\
			&= \sum_{n=1}^\infty \lip f,\phi_n \rip_\lotg \lip w_n \psi_n, (\sigma(0) +V_{\sigma'}^*)\theta_k \rip_\hkg \\
			&= \sum_{n=1}^\infty \lip f,\phi_n \rip_\lotg \lip w_n \psi_n, p_k \rip_\hkg \\
			&= \lip f,\phi_k \rip_\calh.
	\end{align*}
\end{proof}

\begin{remark}
Moreover, $\{\theta_k\}$ is also a Riesz sequence. This follows from the fact that $(I+V_K^*)$ is bounded with a bounded inverse so
	\[ |\sigma(0)|\left\| \sum a_k\theta_k \right\| = \left\| (I+V_K^*) \sum a_kp_k \right\| \asymp \left\| \sum a_kp_k \right\|. \]
Finally, any sequence which is biorthogonal to a Riesz sequence must also be a Riesz sequence \cite[p. 36]{young01}. $\{p_k\}$ is biorthogonal to $\{w_n\psi_n\}$ by construction.
\end{remark}

\begin{remark}\label{remark:1}
The $\hkpg$-norm in the lower inequality in Theorem \ref{thm:stability} cannot be replaced by $\hkg$. 
\end{remark}
\begin{proof}
Assume the inequality can be improved. Then by (\ref{eq:u}), $\{y_n\psi_n\}$
forms a Riesz sequence in $\hkg$ where
	\[ y_n(t) = \int_0^t\sigma(t-s)w_n(s) \, ds. \] 
However, in the case of no memory ($M=0$ in (\ref{viscoode})), for $n \in J_1$, $w_n(t) = \sin(\lambda_nt)$ in which case
	\[ \int_0^t \sigma(t-s) \sin(\lambda_ns) \,ds = -\dfrac{1}{\lambda_n}\left(\sigma(0)\cos(\lambda_nt) + \int_0^t \sigma'(t-s) \cos(\lambda_ns) \, ds \right)\]
so $\|y_n\|_{L^2[0,T]} \le C|\lambda_n|^{-1}$. Since $\{z_n\psi_n\}$ is also a Riesz sequence, taking $a_n = \delta_{n,m}$ and applying the upper inequality, we get ($z_n(t) = e^{i\lambda_n t}$)
	\[ T\|\psi_m\|_\calg^2 = \int_0^T \|e^{i\lambda_m t}\psi_m\|_\calg^2 \, dt= \left\|\sum a_n z_n \psi_n\right\|_\hkg^2 \le C. \]
Therefore, $\|\psi_n\| \le C$ which implies
	\begin{equation}\label{eq:yn-decay} \|y_n\psi_n\|_\hkg \le \dfrac{C}{|\lambda_n|}. \end{equation}
However, if $\{y_n\psi_n\}$ was a Riesz sequence, then taking $a_n = \delta_{n,m}$
	\[ \|y_m\psi_m\|_\hkg = \left\| \sum a_n y_n \psi_n \right\|_\hkg \ge c \]
which contradicts (\ref{eq:yn-decay}) since $|\lambda_n| \to \infty$.
\end{proof}

\bibliographystyle{plain}
\bibliography{/home/waton/mega/School/Research/refs-all/refs-all}

\end{document}